\documentclass[11pt,leqno]{amsart}
\usepackage{amsmath,amssymb,amsthm}
\newtheorem{theorem}{Theorem}[section]
\newtheorem{lemma}[theorem]{Lemma}

\newtheorem{proposition}[theorem]{Proposition}
\newtheorem{corollary}[theorem]{Corollary}
\theoremstyle{definition}

\theoremstyle{remark}

\numberwithin{equation}{section}
\def\fnote#1{\footnote}

\def\real{{\mathbb R}}

\def\ignora#1{}
\def\n3#1{\left\vert  \! \left\vert \! \left\vert \, #1 \, \right\vert \!
  \right\vert \! \right\vert }

\begin{document}

\title{ Octahedral norms in spaces of operators}
%\author{}
%\address{Universidad de Granada, Facultad de Ciencias.
%Departamento de Matem\'{a}tica Aplicada, 18071-Granada (Spain)}

\author{Julio Becerra Guerrero, Gin{\'e}s L{\'o}pez-P{\'e}rez and Abraham Rueda Zoca}
\address{Universidad de Granada, Facultad de Ciencias.
Departamento de An\'{a}lisis Matem\'{a}tico, 18071-Granada
(Spain)} \email{juliobg@ugr.es, glopezp@ugr.es,
arz0001@correo.ugr.es}

\thanks{Partially supported by MEC (Spain) Grant MTM2006-04837 and Junta de Andaluc\'{\i}a Grants FQM-185}
\subjclass{46B20, 46B22. Key words:
  octahedral norms, spaces of operators, projective tensor product}

\maketitle \markboth{J. Becerra, G. L\'{o}pez and A. Rueda
}{Octahedral norms in spaces of operators }

\bigskip

\begin{abstract} We study  octahedral norms in the space of bounded linear operators between Banach
spaces. In fact, we prove that $L(X,Y)$ has octahedral norm
whenever $X^*$ and $Y$ have octahedral norm. As a consequence the
space of operators from $L(\ell_1 ,X)$ has octahedral norm if, and
only if, $X$ has octahedral norm. These results also  allows us to
get the stability of strong diameter 2 property for projective
tensor products of Banach spaces, which is an improvement of   the
known results about the size of nonempty relatively weakly open
subsets in the unit ball of the projective tensor product of
Banach spaces.

\end{abstract}

\section{Introduction.}

Octahedral norms were introduced by G. Godefroy in \cite{G} and
they have been used in order to characterize when a Banach space
contains an isomorphic copy of $\ell_1$. In fact, G. Godefroy
shows in \cite{G} that a Banach space $X$ contains an isomorphic
copy of $\ell_1$ if, and only if, $X$ can be equivalently renormed
so that the new norm in $X$ is octahedral. Also, it is proved in
\cite{BeLoRu} that the norm of a Banach space $X$ is octahedral
if, and only if, every convex combination of $w^*$-slices in the
unit ball of $X^*$ has diameter $2$. As a consequence, the norm of
$X^*$ is octahedral if, and only if, every convex combination of
slices in the unit ball of $X$ has diameter $2$.

In the last years intensive efforts have been done in order to
discover new families of Banach spaces satisfying that every slice
or every nonempty relatively weakly open subset in  its  ball has
diameter 2, (see \cite{ALN}, \cite{BeLo}, \cite{NyWe}, \cite{Sh}).
It is known  that every nonempty relatively weakly open subset in
the unit ball of a Banach space contains a convex combination of
slices. Then having every convex combination of slices with
diameter 2 implies having every slice or every nonempty relatively
weakly open subsets with diameter two. We say that a Banach space
has the strong diameter two property (SD2P) if every convex
combination of slices in its unit ball has diameter 2.

Let us observe that having SD2P implies failing in an extreme way
the well known Radon-Nikodym property, since this property is
characterized in terms of the existence of slices with diameter
arbitrarily small. Also it is known that the projective tensor
product of Banach spaces with the RNP is not necessarily a RNP
space \cite{BP}. However, the unit ball of a such product has
slices with diameter arbitrarily small.  It is proved in
\cite{ABR} that every nonempty relatively weakly open subset in
the unit ball of
 the projective
 tensor product of two Banach spaces with infinite-dimensional
 centralizer has diameter $2$, a property weaker that SD2P. Then it is a natural open question if the
 SD2P is stable for projective tensor products \cite{ALN}.

 The aim of this note is to study the existence of
 octahedral norms in spaces of operators. In this setting we prove that the space of operators
 $L(X,Y)$ has octahedral norm whenever $X$ and $Y^*$ have it. Also we get some necessary conditions in order to
 $L(X,Y)$ has octahedral norm. Now, taking into account the dual
 relation between having octahedral norm and satisfying SD2P we
 show, as a consequence,
 that the SD2P is stable for projective tensor products of Banach spaces, which
 improves the stability  known results in \cite{ABR}. Also, we get
 necessary
 conditions on Banach spaces in order that the projective tensor
 product of these spaces has the SD2P.

 We pass now to introduce some notation. We consider real Banach
 spaces. $B_X$ and $S_X$ stand for the closed unit ball and the
 unit sphere of the Banach space $X$. $X^*$ is the topological
 dual space of $X$. A slice in the unit ball of $X$ is the set
 $$S(B_X , x^* ,\alpha):= \{x\in B_X:\ x^*(x)>1-\alpha\},$$ where $x^*\in S_{X^*}$ and
 $0<\alpha <1$.

 The norm of a Banach space $X$ is said to be {\it octahedral} if,
 for every finite-dimensional subspace $E$ of $X$ and for every
 $\varepsilon>0$ there is $y\in S_X$ such that $$\Vert x+\lambda
 y\Vert \geq (1-\varepsilon)(\Vert x\Vert +\vert \lambda\vert )$$
 for every $x\in E $ and scalar $\lambda$.

 Let $X,Y$ be Banach spaces. According to \cite{DeFlo},we recall that
the {\it{projective tensor product of $X$ and $Y$}}, denoted by
$X\widehat{\otimes}_\pi Y$, is the completion of $X\otimes Y$
under the norm given by

$$\Vert u\Vert:=\inf \left \{\sum_{i=1}^n \Vert x_i\Vert\Vert y_i\Vert\ /\ n\in\mathbb N, x_i\in X, y_i\in Y\ \forall i\in\{1,\ldots, n\}, u=\sum_{i=1}^n x_i\otimes y_i \right\}.$$

We recall that the space $\mathcal B(X\times Y)$ of bounded
bilinear forms defined on $X\times Y$ is linearly isometric to the
topological dual of $X\widehat{\otimes}_\pi Y$.

Let $X$ be a Banach space and $L(X)$ the space of all bounded and
linear operators on $X$. By a {\it multiplier} on $X$ we mean an
element $T \in L(X)$ such that every extreme point of $B_{X^*}$
becomes an eigenvector for $T^t$. Thus, given a multiplier $T$ on
$X$, and an extreme point $p$ of $B_{X^*}$, there exists a unique
scalar  $a_T(p)$ satisfying $ T^t (p) =a_T(p)p$. The  {\it
centralizer} of $X$ (denoted by $Z(X)$) is defined as the set of
those multipliers $T$ on $X$ such that there exists a multiplier
$S$ on $X$ satisfying $a_S(p)=\overline{a_T(p)}$ for every extreme
point $p$ of $B_{X^*}$. Thus, if $X$ a real Banach space, then
$Z(X)$ coincides with the set of all multipliers on $X$. In any
case, $Z(X)$ is a closed subalgebra of $ L (X)$ isometrically
isomorphic to $\mathcal C (K_X)$, for some
%(essentially unique)
compact Hausdorff topological space $K_X$ (see \cite[Proposition
3.10]{Beh}).

Given a Banach space $X$, we consider the increasing sequence of
its even duals
$$
X\subseteq X^{**}\subseteq X^{(4}\subseteq \cdot \cdot \cdot
\subseteq X^{(2n} \subseteq \cdot \cdot \cdot .
$$
Since every Banach space is isometrically embedded into its second
dual, we can define $X^{(\infty }$ as the completion of the normed
space $\bigcup _{n=0}^{\infty }X^{(2n}$.

For a Banach space $X$,  an {\it $L$-projection} on $X$ is a
(linear) projection $P: X \longrightarrow X$ satisfying $ \Vert x
\Vert = \Vert  P(x) \Vert + \Vert x- P(x) \Vert$  for every $x \in
X$.  In such a case, we will say that the subspace $P(X)$ is an
{\it $L$-summand} of $X$. Let us notice that the composition of
two $L$-projections on $X$ is an $L$-projection
\cite[Proposition~1.7]{Beh}, so the closed linear subspace of $L
(X)$ generated by all $L$-projections on $X$ is a subalgebra of $
L (X)$, the space of all bounded and linear operators on $X$. This
algebra, denoted by $C(X)$, is called the {\it Cunningham algebra}
of $X$. It is known that $C(X)$ is linearly isometric  to $Z(X^*)$
(\cite[Theorems 5.7 and 5.9]{Beh}). For instance, any
infinite-dimensional space $L_1 (\mu)$ satisfies that its
Cunningham algebra is infinite-dimensional. By \cite[Lemma
VI.1.1]{HWW}, the centralizer of $L(X,Y)$, the space of all
bounded and linear operators from $X$ to $Y$, is infinite
dimensional whenever  either $C(X)$ is infinite dimensional or
$Z(Y)$ is infinite dimensional.

\section{Main results.}

Our start point is the next proposition, which will be used in
order to deal with spaces satisfying the strong diameter two
property.

\begin{proposition}\label{prstrong}

Let $X$ be a Banach space and $C:=\sum_{i=1}^n \lambda_i S_i$ a
convex combination of slices of $B_X$ such that

$$diam(C)=2.$$

Then for every $\varepsilon>0$ there exist $x_i, y_i\in S_i\ \
\forall i\in\{1,\ldots, n\}$ and $f\in S_{X^*}$ such that

$$f(x_i-y_i)>2-\varepsilon\ \ \forall i\in \{1,\ldots, n\}.$$
Thus $$f(x_i),\ f(-y_i)>1-\varepsilon \ \forall i\in \{1,\ldots,
n\}.$$

\end{proposition}

\begin{proof}

Fix an arbitrary $\varepsilon>0$. Choose $\delta>0$ such that
$\delta<m\varepsilon$, for $m:=\min\limits_{1\leq i\leq
n}\lambda_i$ (notice that we can assume that $\lambda_i\neq 0\
\forall i\in \{1,\ldots, n\}$).

Since $diam(C)=2$ then for every $1\leq i\leq n$ there exist
$x_i,\ y_i\in S_i$  such that
$$\left \Vert \sum_{i=1}^n \lambda_i x_i-\sum_{i=1}^n \lambda_i
y_i\right \Vert >2-\delta .$$ Hence there exists $f\in S_{X^*}$
satisfying

$$f\left (\sum_{i=1}^n \lambda_i x_i-\sum_{i=1}^n \lambda_i y_i \right )=\sum_{i=1}^n \lambda_i f(x_i-y_i)>2-\delta.$$

As a consequence we have

\begin{equation}\label{puntoslej}
f(x_i-y_i)>2-\varepsilon\ \ \forall i\in \{1,\ldots, n\}.
\end{equation}

Indeed, assume that there exists $i\in \{1,\ldots,n\}$ such that
$f(x_i-y_i)\leq 2-\varepsilon$ then

$$2-\delta<\sum_{j=1}^n \lambda_j f(x_j-y_j)=\lambda_i f(x_i-y_i)+\sum_{j\neq i}
\lambda_i f(x_j-y_j)\leq$$
$$\lambda_i(2-\varepsilon)+2(1-\lambda_i)=
2- \lambda_i\varepsilon<2-\delta,$$

a contradiction. So (\ref{puntoslej}) holds.

\end{proof}

We omit the proof of next proposition which is the dual version of
the above one.

\begin{proposition}\label{prstrong-dual}
Let $X$ be a Banach space and $C:=\sum_{i=1}^n \lambda_i S_i$ a
convex combination of $w^*$-slices of $B_{X^*}$ such that
$$diam(C)=2.$$
Then for every $\varepsilon>0$ there exist $f_i, g_i\in S_i\ \
\forall i\in\{1,\ldots, n\}$ and $x\in S_{X}$ such that

$$(f_i-g_i)(x)>2-\varepsilon\ \ \forall i\in \{1,\ldots, n\}.$$
Thus $$f_i(x),\ g_i(-x)>1-\varepsilon\ \forall i\in \{1,\ldots,
n\}.$$

\end{proposition}

Our last preliminary lemma allows us writing the unit ball of the
dual of a space of operators in terms of elements in a tensor
product of spheres.

\begin{lemma}\label{densidad}
Let $X,Y$ be Banach spaces. We consider $H$ a closed subspace of
$L(X,Y)$ such that $X^*\otimes Y\subseteq H$. Then we have that
$B_{H^*}=\overline{co}^{w^*}(S_X\otimes S_{Y^*})$.
\end{lemma}

\begin{proof}
Given $T\in H$, it is clear that $$ \Vert T\Vert =\sup
\{y^*(T(x)):x\in S_X, \ y^*\in S_{Y^*}\}.$$ This  implies that the
set of continuous linear functional $x\otimes y^* \in H^*$ given
by $(x\otimes y^*)(T):=y^*(T(x))$, for $x\in S_X$ and $ y^*\in
S_{Y^*}$, is a norming subset of $H^*$. Since $X^*\otimes
Y\subseteq H$, we have that $\Vert x\otimes y^*\Vert =1$ for every
$x\in S_X$ and $ y^*\in S_{Y^*}$. By a separation argument, we get
that $B_{H^*}=\overline{co}^{w^*}(S_X\otimes S_{Y^*})$.
\end{proof}

Note that as an easy consequence of the above lemma we get that
every $w^*-$slice of the unit ball in $L(X,Y)^*$ has diameter $2$,
whenever every slice of the unit ball in $X$ has diameter $2$.

Let $X$ be a Banach space, and let $u$ be a norm-one element in
$X$. We put $$D(X,u):=\{f\in B_{X^*}:f(u)=1\}.$$ Now, assume that
$X$ has a (complete) predual $X_*$, and put
$$D^{w^*}(X,u):=D(X,u)\cap X_*.$$ If $D^{w^*}(X,u)=\emptyset $, then
we define $n^{w^*}(X,u):=0$. Otherwise, we define $n^{w^*}(X,u)$
as the largest non-negative real number $k$ satisfying $$k\Vert
x\Vert \leq v^{w^*}(x):=\sup \{\vert f(x)\vert :f\in
D^{w^*}(X,u)\}$$ for every $x\in X$. We say that $u$ is a
$w^*$-unitary element of $X$ if the linear hull of $D^{w^*}(X,u)$
equals the whole space $X^*$.

Now, let $X$ be an arbitrary Banach space, and let $u$ be a
norm-one element in~$X$. We define $n(X,u)$ as the largest
non-negative real number $k$ satisfying $$k\Vert x\Vert \leq
v(x):=\sup \{\vert f(x)\vert :f\in D(X,u)\}$$ for every $x\in X$,
and we say that $u$ is a unitary element of $X$ if the linear hull
of $D(X,u)$ equals the whole space $X^*$. Noticing that
$D(X,u)=D^{w^*}(X^{***},u)$, it is clear that $u$ is unitary in
$X$ if and only if it is $w^*$-unitary in $X^{**}$, and that
$n(X,u)=n^{w^*}(X^{**},u)$.

The next result is a sufficient condition in order to have
octahedral norm in a space of operators.

\begin{theorem}\label{teostrong}
Let $X,Y$ be Banach spaces. Assume that $Y$ has octahedral norm
and that there exists $f\in S_{X^{*}}$ such that $n(X^{*},f)=1$.
Let be $H$ a closed subspace of $L(X,Y)$ such that $X^*\otimes
Y\subseteq H$. Then $H$ has octahedral norm.
\end{theorem}

\begin{proof} By \cite[Theorem 2.1]{BeLoRu}, $H$ has octahedral norm if and only
if every convex combination of $w^*$-slices of $H^*$ has diameter
$2$. Let $C:=\sum_{i=1}^n \lambda_i S_i$ a convex combination of
$w^*$-slices in $B_{H^*}$. We can assume that there exist
$\varepsilon\in\real ^+$ and $A_i\in S_{H}$ such that
$S_i=S(B_{H^*},A_i,\varepsilon)$ for every $i\in\{1,\ldots, n\}$.

Fix $\delta\in \real ^+$. By lemma \ref{densidad} we can assume
that there exists $x_i\in S_X$,  $y_i^*\in S_{Y^*}$ such that
$x_i\otimes y_i^*\in S_i$ for all $i\in \{1,\ldots, n\}$. Thus
$\sum_{i=1}^n \lambda_i x_i\otimes y_i^*\in C$.

As $B_X=\overline{\vert co\vert}(D^{w^*}(X^*,f))$ \cite[Corollary
3.5]{ABR} we can assume that $x_i\in D^{w^*}(X^*,f)\ \forall
i\in\{1,\ldots, n\}$.

Fix $i\in\{1,\ldots, n\}$. For $y^*\in B_{Y^*}$ we have
$x_i\otimes y^*\in S_{H^*}$. Hence

$$x_i\otimes y^*\in S_i\Leftrightarrow  A_i(x_i\otimes y^*)>1-\varepsilon\Leftrightarrow A_i(x_i)(y^*)>1-\varepsilon.$$
 Thus

$$x_i\otimes y^*\in S_i\Leftrightarrow y^*\in S(B_{Y^*},A_i(x_i),\varepsilon).$$

Since $\sum_{i=1}^n \lambda_i S(B_{Y^*},A_i(x_i),\varepsilon)$ is
a convex combination of $w^*$-slices in $B_{Y^*}$ we deduce by
corollary \ref{prstrong-dual} and since $Y$ has octahedral norm
the existence of $u_i^*,v_i^*\in B_{Y^*}$ such that $x_i\otimes
u_i^*, x_i\otimes v_i^*\in S_i$ for $i\in \{1,\ldots, n\}$, and
$y\in S_{Y}$ such that

$$y(u_i^*-v_i^*)>2-\delta\ \ \ \forall i\in\{1,\ldots, n\}.$$

Then $\sum_{i=1}^n \lambda_i x_i\otimes u_i^*, \sum_{i=1}^n
\lambda_i x_i\otimes v_i^*\in C$.

Define $T:X\rightarrow Y$ by $T(x):=f(x)y$ for $x\in X$. Clearly
$\Vert T\Vert =1$ and $T\in H$. Thus

$$diam(C)\geq \left \Vert \sum_{i=1}^n \lambda_i x_i\otimes u_i^*-\sum_{i=1}^n \lambda_i x_i\otimes v_i^* \right \Vert=
\left \Vert \sum_{i=1}^n \lambda_i x_i\otimes (u_i^*-v_i^*)\right
\Vert\geq$$
$$\geq \sum_{i=1}^n \lambda_i T(x_i)(u_i^*-v_i^*)=\sum_{i=1}^n \lambda_i f(x_i)y(u_i^*-v_i^*)>(2-\delta)\sum_{i=1}^n \lambda_i =2-\delta.$$

From the arbitrariness of $\delta$ we deduce that $diam(C)=2$ and
the theorem is proved.
\end{proof}

From the symmetry in the proof of the above theorem we get the
following

\begin{theorem}\label{teostrongbis}
Let $X,Y$ be Banach spaces. Assume that $X^*$ has octahedral norm
and that there exists $f\in S_{Y}$ such that $n(Y,f)=1$. Let be
$H$ a closed subspace of $L(X,Y)$ such that $X^*\otimes Y\subseteq
H$. Then $H$ has octahedral norm.
\end{theorem}

The sufficient condition obtained in the above theorem is
satisfied when one assume an infinite-dimensional condition on the
Cunningham algebra.

\begin{corollary}\label{corocunin}
Let $X,Y$ be Banach spaces. Assume that $C(Y^{(\infty})$ is
infinite-dimensional and that there exists $f\in S_{X^{*}}$ such
that $n(X^{*},f)=1$. Let be $H$ a closed subspace of $L(X,Y)$ such
that $X^*\otimes Y\subseteq H$. Then $H$ has octahedral norm.
\end{corollary}

\begin{proof} Since $C(Y^{(\infty})$ is
infinite-dimensional, we have that $Z((Y^{(\infty})^*)$ is
infinite-dimensional. By \cite[Theorem 3.4]{ABL}, $Y^*$ has the
strong diameter two property, so theorem \ref{teostrong} applies.
\end{proof}

Keeping in mind the projective tensor product, we improve the
thesis of \cite[Theorem 3.6]{ABR} invoking theorem
\ref{teostrong}.

\begin{corollary}
Let $X,Y$ be Banach spaces. Assume that $Z(Y^{(\infty})$ is
infinite-dimensional and that there exists $f\in S_{X^{*}}$ such
that $n(X^{*},f)=1$. Then $X\widehat{\otimes}_\pi Y$ has the
strong diameter two property.
\end{corollary}

The next result follows the lines of  the theorem \ref{teostrong},
but for another kind of spaces and we get, as a consequence, the
complete stability of octahedral norms in spaces of operators.

\begin{theorem}\label{teostrong2}
Let $X,Y$ be Banach spaces. Assume that $X^*$ and $Y$ have
octahedral norm. Let be $H$ a closed subspace of $L(X,Y)$ such
that $X^*\otimes Y\subseteq H$. Then $H$ has octahedral norm.
\end{theorem}

\begin{proof}
Again by \cite[Theorem 2.1]{BeLoRu}, $H$ has octahedral norm if
and only if every convex combination of $w^*$-slices of $H^*$ has
diameter $2$. Let $C:=\sum_{i=1}^n \lambda_i S_i$ a convex
combination of $w^*$-slices in $B_{H^*}$. Hence there exist
$\varepsilon\in\real ^+$ and $A_i\in S_{H}$ such that
$S_i=S(B_{H^*},A_i,\varepsilon)$ for every $i\in\{1,\ldots, n\}$.
Fix $\delta\in \real ^+$.

By lemma \ref{densidad} we can ensure the existence of $x_i\in
S_X$,  $y_i^*\in S_{Y^*}$ such that $x_i\otimes y_i^*\in S_i$ for
all $i\in \{1,\ldots, n\}$. Thus $\sum_{i=1}^n \lambda_i
x_i\otimes y_i^*\in C$. For every $i\in\{1,\ldots, n\}$ we
consider  $A^*_i(y_i^*):X\rightarrow \real$ defined by
$A_i^*(y_i^*)(x)=y_i^*(A_i(x))$ for every $x\in X$. Then we have
that $$x\otimes y_i^*\in S_i\Leftrightarrow x\in
S(B_X,A_i^*(y_i^*), \varepsilon).$$ Now, $X^*$ has octahedral norm
and by \cite[Theorem 2.1]{BeLoRu}, we have that every convex
combination of slices of $X$ has diameter $2$.

Applying Proposition \ref{prstrong}, there exists $w_i\in
S(B_X,A_i^*(y_i^*),\varepsilon)$ and $f\in S_{X^*}$ such that for
all $i\in\{1,\ldots, n\}$

$$f(w_i)>1-\delta,$$

and

$$\sum_{i=1}^n \lambda_i w_i\otimes y_i\in C .$$

Following as in proof of theorem \ref{teostrong}, we deduce that
$diam(C)>2-\delta$. Due to the arbitrariness of $\delta$ we deduce
that $diam(C)=2$ and the theorem is proved.
\end{proof}

The first consequence of the above theorem is the stability of
SD2P for projective tensor products of Banach spaces. Indeed, it
is enough taking into account that the dual of the projective
tensor product of Banach spaces $X$ and $Y$ is the space of
operators $L(X,Y^*)$ joint to the relation between having
octahedral norm and the SD2P, as explained in the introduction,
and apply the above theorem to get the following

\begin{corollary}\label{SD2P}
The projective tensor product,  $X\widehat{\otimes}_\pi Y$, of two
Banach spaces $X$ and $Y$, verifies the SD2P whenever $X$ and $Y$
have the SD2P.
\end{corollary}

This last corollary gives the stability of SD2P for projective
tensor products. However, in order to complete the study of the
behavior of SD2P for projective tensor products it would be
interesting to know if the projective tensor product of two Banach
spaces has SD2P when one assumes that only one of the spaces has
the SD2P. We think that a possible candidate to answer by the
negative this question could be $X\widehat{\otimes}_\pi \ell_2^2$,
for some Banach space $X$ with the SD2P, but we don't know the
answer.

We recall, from proof of corollary \ref{corocunin}, that given
Banach spaces $X$ and $Y$ such that $Z(X^{(\infty})$ and
$C(Y^{(\infty})$ are infinite-dimensional, then both $X^*$ and $Y$
have octahedral norm.

\begin{corollary}\label{mejor-2}
Let $X,Y$ be Banach spaces such that $Z(X^{(\infty})$ and
$C(Y^{(\infty})$ are infinite-dimensional. Let be $H$ a closed
subspace of $L(X,Y)$ such that $X^*\otimes Y\subseteq H$. Then $H$
has octahedral norm.
\end{corollary}

As a new consequence, we get a result improving \cite[Theorem
2.6]{ABR}, where from the same hypotheses is obtained that every
nonempty relatively weakly open subset of the unit ball
 in $X\widehat{\otimes}_\pi Y$ has diameter $2$.

\begin{corollary}\label{mejor}

Let $X,Y$ be Banach spaces such that $Z(X^{(\infty})$ and
$Z(Y^{(\infty})$ are infinite-dimensional. Then the space
$X\widehat{\otimes}_\pi Y$ has the strong diameter two property.

\end{corollary}

Let $X$ be a Banach space. For $u\in S_X$ we define the
{\it{roughness of $X$ at $u$}}, denoted by $\eta(X,u)$, by the
equality

$$\eta(X,u)=\limsup\limits_{\Vert h\Vert\rightarrow 0} \frac{\Vert u+h\Vert+\Vert u-h\Vert-2}{\Vert h\Vert}$$

Let us remark that, for $u\in S_X$, we have that $\eta(X,u)=0$ if,
and only if, the norm is Fr\'echet differentiable at $u$
\cite[Lemma 1.3]{DGZ}.

For $\varepsilon\in\mathbb R^+$, the Banach space is said to be
{\it{$\varepsilon$-rough}} if, for every $u\in S_X$ we have
$\eta(X,u)\geq \varepsilon$. We say that $X$ is {\it{rough}}
whenever it is $\varepsilon$-rough for some $\varepsilon\in\mathbb
R^+$, and we say that $X$ is {\it{non-rough}} otherwise.

Our next result is a necessary condition in order to a space of
operators has an octahedral norm.

\begin{proposition}\label{necesaria} Let $X,Y$ be Banach spaces and let be $H$ a closed subspace of $L(X,Y)$ with octahedral norm such that $X^*\otimes
Y\subseteq H$. Assume that $X^*$ has non-rough norm. Then $Y$ has
octahedral norm.
\end{proposition}
\begin{proof} We prove that every convex combination of
$w^*$-slices of $B_{Y^*}$ has diameter $2$. We put $y_1,\ldots
,y_n\in S_Y$, $\delta >0$ and $\lambda_1,\ldots ,\lambda_n\in
(0,1)$ with $\sum_{i=1}^{n}\lambda_i=1$, and consider the convex
combination of $w^*$-slices
$$\sum_{i=1}^{n}\lambda_iS(B_{Y^*},y_i,\delta ).$$ Let be
$\varepsilon >0$. Since the norm of $X^*$ is non-rough, from
\cite[Proposition 1.11]{DGZ}, we have that there exist $x^*\in
S_{X^*}$ and $\alpha >0$ such that $diam (S(B_X,x^*,\alpha
))<\varepsilon .$ Put $\rho :=\min \{\delta ,\alpha\}$ and $x_0\in
S_X\cap S(B_X,x^*,\alpha) $. Consider the convex combination of
$w^*$-slices of $B_{H^*}$ given by
$$\sum_{i=1}^{n}\lambda_iS(B_{H^*},x^*\otimes y_i,\rho^2 ).$$ Now,
$H$ has octahedral norm, then for $i\in \{1,\ldots ,n\}$ there
exist $f_i,g_i\in S_{H^*}\cap S(B_{H^*},x^*\otimes y_i,\rho^2 )$
such that

$$\left\Vert \sum_{i=1}^{n}\lambda_if_i-\sum_{i=1}^{n}
\lambda_ig_i\right\Vert
>2-\varepsilon .$$

By lemma \ref{densidad} we can assume that
$f_i=\sum_{k=1}^{m_i}\gamma_{(k,i)} x_{(k,i)}\otimes y_{(k,i)}^*$
and $g_i=\sum_{k=1}^{m_i}\gamma_{(k,i)}' u_{(k,i)}\otimes
v_{(k,i)}^*$, where $x_{(k,i)},u_{(k,i)}\in S_X$,
$y_{(k,i)}^*,v_{(k,i)}^*\in S_{Y^*}$, and
$\sum_{k=1}^{m_i}\gamma_{(k,i)}
=1=\sum_{k=1}^{m_i}\gamma_{(k,i)}'$ which $\gamma_{(k,i)},
\gamma_{(k,i)}'\in [0,1]$ for all $(k,i), k\in \{1,\ldots ,m_i\}$
and $i\in \{1,\ldots ,n\}$. For $i\in \{1,\ldots ,n\}$, we
consider the sets $P_i:=\{(k,i)\in \{1,\ldots ,m_i\}\times\{i\}:
(x^*\otimes y_i)(x_{(k,i)}\otimes y_{(k,i)}^*)>1-\rho\}$ and
$Q_i:=\{(k,i)\in \{1,\ldots ,m_i\}\times\{i\}: (x^*\otimes
y_i)(u_{(k,i)}\otimes v_{(k,i)}^*)>1-\rho\}$. Then we have that

$$1-\rho^2<(x^*\otimes y_i)\left(\sum_{k=1}^{m_i}\gamma_{(k,i)} x_{(k,i)}\otimes y_{(k,i)}^*\right)=$$

$$ \sum_{i\in P_i}\gamma_{(k,i)} (x^*\otimes y_i)(x_{(k,i)}\otimes y_{(k,i)}^*)+\sum_{i\notin P_i}\gamma_{(k,i)}
(x^*\otimes y_i)(x_{(k,i)}\otimes y_{(k,i)}^*)\leq$$

$$\sum_{i\in P_i}\gamma_{(k,i)}+(1-\rho )\sum_{i\notin P_i}\gamma_{(k,i)}=1-\sum_{i\notin P_i}
\gamma_{(k,i)}+(1-\rho )\sum_{i\notin P_i}\gamma_{(k,i)}.$$

 We
conclude that $\sum_{i\notin P_i}\gamma_{(k,i)}<\rho$, and hence
we have that

$$(x^*\otimes y_i)\left(\sum_{i\in P_i}\gamma_{(k,i)}
x_{(k,i)}\otimes y_{(k,i)}^*\right)>1-\rho .$$

It follows that
$$y_i\left(\sum_{i\in P_i}\gamma_{(k,i)}
y_{(k,i)}^*\right)\geq (x^*\otimes y_i)\left(\sum_{i\in
P_i}\gamma_{(k,i)} x_{(k,i)}\otimes y_{(k,i)}^*\right)>1-\rho ,$$

 and $\sum_{i\in
P_i}\gamma_{(k,i)} y_{(k,i)}^*\in B_{Y^*}$, so
$\varphi_i:=\sum_{i\in P_i}\gamma_{(k,i)} y_{(k,i)}^*\in
S(B_{Y^*},y_i,\delta )$.

For $(k,i)\in P_i$, we have that $(x^*\otimes
y_i)(x_{(k,i)}\otimes y_{(k,i)}^*)>1-\rho$. This implies that
$x^*(x_{(k,i)})>1-\rho$, and as a consequence $\Vert
x_{(k,i)}-x_0\Vert <\varepsilon$. In a similar way, we have that

$$y_i\left(\sum_{i\in Q_i}\gamma_{(k,i)}'
y_{(k,i)}^*\right)\geq (u^*\otimes v_i)\left(\sum_{i\in
Q_i}\gamma_{(k,i)}' u_{(k,i)}\otimes v_{(k,i)}^*\right)>1-\rho ,$$
and $\sum_{i\in Q_i}\gamma_{(k,i)}' v_{(k,i)}^*\in B_{Y^*}$, so
$\psi_i:=\sum_{i\in Q_i}\gamma_{(k,i)} v_{(k,i)}^*\in
S(B_{Y^*},y_i,\delta )$.

For $(k,i)\in Q_i$, we have that $(x^*\otimes
y_i)(u_{(k,i)}\otimes v_{(k,i)}^*)>1-\rho$. This implies that
$x^*(u_{(k,i)})>1-\rho$, and as a consequence $\Vert
u_{(k,i)}-x_0\Vert <\varepsilon$. It follows that

$$\Vert f_i-x_0\otimes\varphi_i\Vert \leq \left\Vert f_i-\sum_{i\in P_i}\gamma_{(k,i)}
x_{(k,i)}\otimes y_{(k,i)}^* \right\Vert +\left\Vert \sum_{i\in
P_i}\gamma_{(k,i)} x_{(k,i)}\otimes
y_{(k,i)}^*-x_0\otimes\varphi_i\right\Vert =$$

$$\left\Vert \sum_{i\notin P_i}\gamma_{(k,i)}
x_{(k,i)}\otimes y_{(k,i)}^* \right\Vert +\left\Vert \sum_{i\in
P_i}\gamma_{(k,i)} (x_{(k,i)}-x_0)\otimes y_{(k,i)}^*\right\Vert
\leq
$$

$$ \sum_{i\notin P_i}\gamma_{(k,i)}
+ \sum_{i\in P_i}\gamma_{(k,i)} \Vert x_{(k,i)}-x_0\Vert \Vert
y_{(k,i)}^*\Vert \leq \rho +\varepsilon .
$$ In a similar way, we have that
$$\Vert g_i-x_0\otimes \psi_i\Vert \leq \rho +\varepsilon .
$$

As a consequence we have that

$$2-\varepsilon < \left\Vert
\sum_{i=1}^{n}\lambda_if_i-\sum_{i=1}^{n}\lambda_ig_i\right\Vert
\leq
$$
$$ \left\Vert
\sum_{i=1}^{n}\lambda_i(f_i-x_0\otimes \varphi_i)\right\Vert
+\left\Vert \sum_{i=1}^{n}\lambda_ix_0\otimes (\varphi_i-
\psi_i)\right\Vert +\left\Vert
\sum_{i=1}^{n}\lambda_i(g_i-x_0\otimes \psi_i)\right\Vert \leq$$

$$2(\rho +\varepsilon )+ \left\Vert
\sum_{i=1}^{n}\lambda_ix_0\otimes (\varphi_i- \psi_i)\right\Vert
=2(\rho +\varepsilon )+ \left\Vert x_0\otimes
\sum_{i=1}^{n}\lambda_i(\varphi_i- \psi_i)\right\Vert \leq$$

$$ 2(\rho
+\varepsilon )+ \Vert x_0\Vert \left\Vert
\sum_{i=1}^{n}\lambda_i(\varphi_i- \psi_i)\right\Vert \leq 2(\rho
+\varepsilon )+ \left\Vert \sum_{i=1}^{n}\lambda_i(\varphi_i-
\psi_i)\right\Vert .$$

It follows that

$$\left\Vert \sum_{i=1}^{n}\lambda_i\varphi_i-
 \sum_{i=1}^{n}\lambda_i\psi_i\right\Vert > 2-2\rho -3\varepsilon .$$
 We recall that $\varphi_i,\psi_i\in S(B_{Y^*},y_i,\delta )$, and
 hence
 $$diam (\sum_{i=1}^{n}\lambda_iS(B_{Y^*},y_i,\delta ))\geq  2-2\rho -3\varepsilon\geq 2-2\delta-3\varepsilon
 .$$

Since $\varepsilon$ is arbitrary, we conclude that
 $$diam \left(\sum_{i=1}^{n}\lambda_iS(B_{Y^*},y_i,\delta )\right)\geq 2-2\delta.$$

Hence, for $0<\eta<\delta$ we have

$$\sum_{i=1}^{n}\lambda_iS(B_{Y^*},y_i,\eta )\subseteq \sum_{i=1}^{n}\lambda_iS(B_{Y^*},y_i,\delta ),$$

and $diam\left( \sum_{i=1}^{n}\lambda_iS(B_{Y^*},y_i,\eta
)\right)\geq 2-2\eta$ by using a similar argument. Hence

$$diam \left(\sum_{i=1}^{n}\lambda_iS(B_{Y^*},y_i,\delta )\right)\geq 2-2\eta.$$

Since $\eta\in (0,\delta)$ is arbitrary we deduce that

$$diam \left(\sum_{i=1}^{n}\lambda_iS(B_{Y^*},y_i,\delta )\right)=2,$$

and we are done. \end{proof}

From the symmetry of the spaces $X$ and $Y$ in the proof of the
above result, this one can be written also in the following way.

\begin{corollary}\label{simetri} Let $X,Y$ be Banach spaces and let be $H$ a closed subspace of $L(X,Y)$ with octahedral norm
such that $X^*\otimes Y\subseteq H$. Assume that $Y$ has non-rough
norm. Then $X^*$ has octahedral norm. \end{corollary}

As a consequence of the above proposition and the Theorem
\ref{teostrong} one gets the following equivalence.

\begin{corollary}\label{corolario-teostrong}
Let $X,Y$ be Banach spaces. Assume that the norm of $X^*$ is
non-rough and that there exists $f\in S_{X^{*}}$ such that
$n(X^{*},f)=1$. Then for every closed subspace $H$ of $L(X,Y)$
such that $X^*\otimes Y\subseteq H$, the following assertion are
equivalent:
\begin{enumerate}
\item[i)] $H$ has octahedral norm. \item[ii)] $Y$ has octahedral
norm.
\end{enumerate}
\end{corollary}

Taking $X=\ell_1$ in the above corollary one has

\begin{corollary} Let $Y$ be Banach space. Then $L(\ell_1 ,Y)$ has
octahedral norm if and only if $Y$ has octahedral norm.
\end{corollary}

Again, using the duality between having octahedral norm and the
SD2P joint to the duality $(X\widehat{\otimes}_\pi Y)^*
=L(X,Y^*)$, we get from Proposition \ref{necesaria}, a necessary
condition in order to the projective tensor product of Banach
spaces has the SD2P.

\begin{corollary} Assume that $X$ and $Y$ are Banach spaces such
that $X\widehat{\otimes}_\pi Y$ has SD2P and  $X^*$ has non-rough
norm. Then $Y$ has SD2P.
\end{corollary}

Let $X$ be a Banach space. According to \cite{HLP} we say that
(the norm on) $X$ is {\it{weakly octahedral}} if, for every
finite-dimensional subspace $Y$ of $X$, every $x^*\in B_{X^*}$,
and every $\varepsilon\in\mathbb R^+$ there exists $y\in S_X$
satisfying

$$\Vert x+y\Vert\geq (1-\varepsilon)(\vert x^*(x)\vert+\Vert y\Vert)\ \forall x\in Y.$$

It is clear that octahedral norm implies weakly octahedral norm,
but the converse is not true. In fact, by \cite[Theorem
2.1]{BeLoRu} and \cite[Theorem 3.4]{HLP}, $(c_0\oplus_p c_0)^*$
has weakly octahedral norm but does not have octahedral norm
\cite[Theorem 3.2]{ABL} for every $1<p<\infty$. From \cite{HLP},
we know that a Banach space $X$ has a weakly  octahedral norm if,
and only if, every nonempty relatively weak-star open subset of
the unit ball in $X^*$ has diameter $2$.

Now our aim is to enlarge the family of examples of spaces
enjoying weakly octahedral norm but failing to have octahedral
norm.

\begin{proposition}\

Let $p\geq 1$ and $X$ a non-null Banach space. Then
$$Y:=L(c_0\oplus_p c_0,X),$$
has weakly octahedral norm.
\end{proposition}

\begin{proof} By \cite[Theorem 3.4]{HLP} it is equivalent to show that $Y^*$ has
the weak$^*$ diameter 2 property, that is, every nonempty
relatively weak-star open subset of the unit ball in $Y^*$ has
diameter $2$.

So let be $\varepsilon\in\mathbb R^+$ and $W$ be a non-empty
relatively $w^*-$open subset of $B_Y$. As $X\neq \{0\}$ and
$c_0\oplus_p c_0$ is infinite-dimensional, $Y$ is also
infinite-dimensional. Then

$$W\cap S_Y\neq \emptyset.$$

By lemma 2.6 there are $m\in\mathbb N, (f_1,g_1),\ldots,
(f_m,g_m)\in S_{c_0\oplus_p c_0}, x_1^*,\ldots,x_n^*\in S_{X^*},
\lambda_1,\ldots, \lambda_n\in [0,1], \sum_{i=1}^m \lambda_i=1$
such that

$$\sum_{i=1}^m \lambda_i (f_i,g_i)\otimes x_i^*\in W,$$

and $\left \Vert \sum_{i=1}^m \lambda_i (f_i,g_i)\otimes
x_i^*\right\Vert>1-\varepsilon^2$.

As $f_1,\ldots, f_m, g_1,\ldots, g_m\in c_0$ we can assume that
$f_i,g_i$ have finite support for every $i\in\{1,\ldots, n\}$
because $W$ is a norm open set. Hence there exists $k\in\mathbb N$
such that for every   $n\geq k$ one has

$$f_i(n)=g_i(n)=0\ \forall i\in\{1,\ldots,m\}.$$

In other words

\begin{equation}\label{sopofinito}(f_i,g_i)=\sum_{j=1}^k f_i(j) (e_j,0)+\sum_{j=1}^k g_i(j) (0,e_j)\ \ \forall i\in\{1,\ldots, m\}.
\end{equation}

As $\left \Vert \sum_{i=1}^m \lambda_i (f_i,g_i)\otimes
x_i^*\right\Vert>1-\varepsilon^2$ we can find $T\in
S_{L(c_0\oplus_p c_0,X)}$ such that

\begin{equation}\label{operateohanban}\sum_{i=1}^m \lambda_i x_i^*(T(f_i,g_i))>1-\varepsilon.
\end{equation}

Fix $i\in\{1,\ldots,n\}$. Define

$$u_{i,n}:=(f_i,g_i)+\sum_{j=1}^k f_i(j)(e_{nk+j},0)+\sum_{j=1}^k g_i(j) (0,e_{nk+j}), $$ $$v_{i,n}:=(f_i,g_i)-\sum_{j=1}^k f_i(j)(e_{nk+j},0)-\sum_{j=1}^k g_i(j) (0,e_{nk+j}).$$

We claim that $\Vert (f_i,g_i)\Vert=\Vert u_{i,n}\Vert=\Vert
v_{i,n}\Vert\ n\geq 2\ \forall i\in\{1,\ldots, n\}$. Let check
only the first equality because the second one is similar. From
(\ref{sopofinito}) we deduce that

$$\Vert u_{i,n}\Vert=\left ( \left(\max\limits_{1\leq j\leq k} \vert f(j)\vert \right)^p+  \left(\max\limits_{1\leq j\leq k} \vert g(j)\vert \right)^p  \right )^\frac{1}{p}=\left ( \Vert f_i\Vert^p+\Vert g_i\Vert^p\right)^\frac{1}{p}=\Vert (f_i,g_i)\Vert.$$

On the other hand, since $\{e_n\}\rightarrow 0$ in the weak
topology of $c_0$, it is clear that

$$\{u_{i,n}\}_{n\in\mathbb N}\rightarrow (f_i,g_i)\ \forall i\in\{1,\ldots,m\},$$

where the last convergence is in the weak topology of $c_0\oplus_p
c_0$.

In a similar way $\{v_{i,n}\}_{n\in\mathbb N}\rightarrow
(f_i,g_i)\ \forall i\in\{1,\ldots,m\}$ in the weak topology of
$c_0\oplus_p c_0$. Hence

$$\{u_{i,n}\otimes x_i^*\}_{n\in\mathbb N}\rightarrow (f_i,g_i)\otimes x_i^*\ \forall i\in\{1,\ldots, n\},$$

in the weak$^*$ topology of $Y^*$.

In order to check the last assertion pick $G\in L(c_0\oplus_p
c_0,X)$. As $G$ is norm-norm continuous then $G$ is weak-weak
continuous. Hence

$$\{G(u_{i,n})\}\rightarrow^w G(f_i,g_i).$$

By definition of weak topology in $X$ then

$$\{u_{i,n}\otimes x_i^*(T)\}=\{x_i^*(T(u_{i,n}))\}\rightarrow x_i^*(T(f_i,g_i))=(f_i,g_i)\otimes x_i^*(T).$$

Last equality proves that $\{u_{i,n}\otimes
x_i^*\}\rightarrow^{w^*} (f_i,g_i)\otimes x_i^*.$

Similarly

$$\{v_{i,n}\otimes x_i\}_{n\in\mathbb N}\rightarrow (f_i,g_i)\otimes x_i\ \forall i\in\{1,\ldots, n\},$$

in the weak$^*$ topology of $Y$. Thus there exist $n\in\mathbb N$
big enough such that

$$u:=\sum_{i=1}^m \lambda_i u_{i,n}\otimes x_i^*,$$

$$v:=\sum_{i=1}^m \lambda_i v_{i,n}\otimes x_i^*,$$

are elements of $W$ (Notice that $\Vert (f_i,g_i)\Vert=\Vert
u_{i,n}\Vert=\Vert v_{i,n}\Vert\ n\geq 2\ \forall i\in\{1,\ldots,
n\}$ implies that $u,v$ are elements in the unit ball of $Y$).
\vspace{0.3cm}

Now our aim is to estimate the norm of the element

$$u-v=2\sum_{i=1}^n \lambda_i \left ( \sum_{j=1}^k f_i(j)(e_{nk+j},0)+\sum_{j=1}^k g_i(j) (0,e_{nk+j})\right ) \otimes x_i^*.$$

Define

$$M:=span \{ (e_{nk+j},0),(0,e_{nk+j}): j\in\{1,\ldots, k\} \},$$

$$ Z:=span \{ (e_{j},0),(0,e_{j}): j\in\{1,\ldots, k\} \}.$$

Then $\Phi:M\longrightarrow Z$ given by

$$\Phi(e_{nk+j},0)=(e_j,0)\ \ \Phi(0,e_{nk+j})=(0,e_j),$$

defines a linear isometry. Moreover it is  clear that

$$\Phi\left ( \sum_{j=1}^k f_i(j)(e_{nk+j},0)+\sum_{j=1}^k g_i(j) (0,e_{nk+j})\right )=(f_i,g_i)\ \forall i\in\{1,\ldots,m\}.$$

Define $P:c_0\oplus_p c_0\longrightarrow c_0\oplus_p c_0$ by the
equation

$$P(f,g)=\sum_{s=nk+1}^{(n+1)k} f(s)(e_s,0)+\sum_{s=nk+1}^{(n+1)k} g(s)(0,e_s).$$

$P$ is a linear projection and clearly $\Vert P\Vert=1$. In
addition, for every $i\in\{1,\ldots, n\}$ we have

$$\Phi\left ( P\left( \sum_{j=1}^k f_i(j)(e_{nk+j},0)+\sum_{j=1}^k g_i(j) (0,e_{nk+j})\right)\right)=$$

$$=\sum_{j=1}^k f_i(j)(e_{j},0)+\sum_{j=1}^k g_i(j) (0,e_{j})=(f_i,g_i).$$

Define

$$S:c_0\oplus_p c_0\longrightarrow X\ \ \ S=T\circ\Phi\circ P.$$

$S\in L(c_0\oplus_p c_0,X)$ and $\Vert S\Vert\leq 1$. If we
compute $(u-v)(S)$ we have

$$(u-v)(S)=2\sum_{i=1}^n \lambda_i x_i^* \left(T\left( \Phi\left( P\left ( \sum_{j=1}^k f_i(j)(e_{nk+j},0)+\sum_{j=1}^k g_i(j) (0,e_{nk+j}) \right)\right)\right )\right)=$$

$$=2\sum_{i=1}^n \lambda_i x_i^*(T(f_i,g_i)).$$

Hence

$$diam(W)\geq \Vert u-v\Vert\geq (u-v)(S)= 2\sum_{i=1}^m \lambda_i x_i^*(T(f_i,g_i))\mathop{>}
\limits^{\mbox{(\ref{operateohanban})}} 2(1-\varepsilon).$$

From the last estimation we deduce that $diam(W)=2$ from the
arbitrariness of $\varepsilon$. \vspace{0.3cm}

By the arbitrariness of $W$ we deduce that $Y$ has weakly
octahedral norm, so we are done.

\end{proof}

Applying  the above result joint to Corollary \ref{simetri} and
\cite[Theorem 3.2]{ABL} we obtain the desired example.

\begin{corollary}\

Let $X$ be a non-null Banach space such that $X$ has non-rough
norm and $p>1$. Then the norm of the space

$$Y=L(c_0\oplus_p c_0,X),$$

is weakly octahedral but not octahedral.

\end{corollary}

\end{document}